\documentclass{amsart}

\title[Positive Entropy Through Pointwise Dynamics]
      {Positive Entropy Through Pointwise Dynamics}

\author[A. Arbieto and E. Rego]{A. Arbieto and E. Rego}
\address{Instituto de Matem\'atica, Universidade Federal do Rio de Janeiro, P. O. Box 68530, 21945-970 Rio de Janeiro, Brazil.}
\email{arbieto@im.ufrj.br}
\email{elias@im.ufrj.br}
\thanks{ A. was partially supported by CNPq, FAPERJ and PRONEX/DS from Brazil. E. R. was partially supported by CAPES from Brazil.}

\newtheorem{theorem}{Theorem}
\newtheorem{corollary}[theorem]{Corollary}

\newtheorem{lemma}[theorem]{Lemma}
\newtheorem{proposition}[theorem]{Proposition}

\theoremstyle{definition}

\newcommand{\de} {\delta}

\newcommand{\si} {\sigma}       \newcommand{\Si}{\Sigma}

\newcommand{\N}{\mathbb{N}}

\newcommand{\eps}{\varepsilon}

\begin{document}

\begin{abstract}
 We define some pointwise properties of topological dynamical systems and give pointwise conditions to guarantee the positiveness of the topological entropy of such a system.
 We also give sufficient conditions to obtain positive topological entropy for maps which are approximated by maps with the shadowing property in a uniform way. 
\end{abstract}

\maketitle

\section{Introduction}

Topological entropy is a quantity that measures the complexity of a dynamical system. In particular its positiveness implies some chaotic behavior of the orbits. Smale in \cite{SM}  constructed a Horseshoe which is an example of a dynamical system with positive entropy. Actually, a Horseshoe appears if there exists a transversal homoclinic point. In particular, this is used in hyperbolic dynamics, to show that any non-trivial hyperbolic basic set has positive entropy, see \cite{ROB}.

Indeed, it is proved that any hyperbolic basic set has a Markov partition. Thus, it is semi-conjugated to a subshift of finite type. If this subshift is actually a full shift then it has positive entropy. So it is a natural question to decide when a dynamical system possessess an invariant set which is semi-conjugated to a full shift. In this paper we give some sufficient conditions to show this using pointwise dynamics.

Some global dynamical properties can be obtained through analogous properties defined on points, this is what we call pointwise dynamics. For instance, transitivity can be obtained by the existence of a point with dense orbit. 

Apart from the topological entropy, some well known dynamical properties are the shadowing property and positive expansiveness.  The former says that trajectories which allow errors are approximated by real ones. The latter says that different orbits must be apart in some future iterate. These properties were studied in the pointwise viewpoint. For instance, in \cite{morales} C.A. Morales gave the definition of shadowable point and  proved that for compact metric spaces the shadowing property is equivalent to all points being shadowable.

For expansiveness, In \cite{RD}, Reddy gives the notion of a positive  expansive point. However, it is not clear whether positive expansiveness of all points implies positive expansivess of the map. Actually, this is false, a counter-example appears in section 5. However, it is possible to define uniform positive expansive points. Once more, if all points are uniformly positive expansive then the map is positive expansive, see section 2.

In this article $(X,d)$ will be a compact metric space and $f:X\to X$ will be a continuous map. We also recall that a periodic point is a point with finite orbit. A wandering point is a point with a neighborhood such that its iterates never intersect it. 

Using the above concepts and some methods developed in  \cite{moop} we obtain our main result which can be stated as follows. 

\begin{theorem}\label{exp}
If there exists a non-periodic point $x$ which is shadowable, uniformly positively expansive, nonwandering and non-isolated, then there exists an invariant closed subset $Y$ of $X$, such that $f|_{Y}$ is topologically conjugated to the shift map. In particular, $f$ has positive topological entropy. 
\end{theorem}

In \cite{morales2}, further expansive-like properties have been studied. Including the concept of $n$-expansivity and countable-expansivity. The first one, allows at most $n$ different orbits to be  close, and the second one allows at most countable different orbits to be close. It is possible to give pointwise versions of these properties. We can modify our arguments in section 3 to obtain the following result. 

\begin{theorem}\label{nexp}
	If there exists a non-periodic point $x$ which is non-isolated, shadowable, positive uniformly $n$-expansive and nonwandering then there exists a closed invariant subset $Y$ of $X$, such that $f|_{Y}$ is topologically semi-conjugated to the shift map. In particular, $f$ has positive topological entropy. 
\end{theorem}

However, the following is still an open question.

{\bf Question:}  Suppose that for every $x\in X$ every neighborhood of $x$ is uncountable. If there exists a non-periodic point $x$ which is shadowable, positive uniformly countable-expansive and nonwandering is it possible to show the existence of a closed invariant subset $Y$ of $X$, such that $f|_{Y}$ is topologically semi-conjugated to the shift map? Is it true that $f$ has positive topological entropy?  

As a consequence of the proof of the previous theorem,  we obtain the following result.

\begin{corollary}\label{per}
	 If $f$ is expansive, has the shadowing property and infinitely many periodic points, then $f$ has positive topological entropy. 
\end{corollary}

We remark that our methods can be used to obtain similar results for homeomorphisms, using expansivity instead positive expansivity.

In \cite{BW}, it is given an example of a 2-expansive map which is not expansive.  As an application, we show that this example has positive entropy despite of the base metric space being zero dimensional. See section 5.

Finally, we give other sufficient conditions to obtain positive entropy if the dynamics is approximated by some maps with the shadowing property in a uniform way. 

\begin{theorem}\label{pentropy}
Let $f_n:X\to X $ be a sequence of nonwandering maps converging uniformly to a map $f$. Suppose that $f_n$ has the uniform shadowing property. Then $f$ has positive entropy.	
\end{theorem}

For this, we use the results developed in \cite{moo} to obtain an entropy point. We also give and use a slight improvement of a result due to Fedeli and Le Donne in \cite{fedeli}.

\section{Preliminaires}

In this section we give precise definitions of the concepts used in the statements of the results. 

Let $(X,d)$ be a metric compact space and $f:X\to X$ be a continuous map.
\vspace{0.1in}

\emph{Topological Entropy}

Let $V$ be a subset of $X$. A set $E\subset V$ is called a $(n,\eps)$-separated set if for every pair of distinct points $x,y\in E$, there exists some $0\leq i\leq n$ such that  $d(f^i(x),f^i(y))\geq \eps$. We set $s_n(V,\eps)$ as  the maximal cardinality of all $(n,\eps)$-separated subsets $E\subset V$.

The \emph{topological entropy} of $f$ is the following quantity
$$h(f)=\lim\limits_{\eps\to 0}\limsup\limits_{n\to\infty}\frac{1}{n}\log s_n(X,\eps).$$
 
This limit always exists, see \cite{WAL}. From this, one can define an entropy point $x$ as a point such that every closed neighborhood $V$ of $x$ satisfies: $$\lim\limits_{\eps\to0}\limsup\limits_{n\to\infty}\frac{1}{n}\log s_n(V,\eps)>0$$

Obviously, the existence of an entropy point guarantees that $f$ has positive topological entropy.

\vspace {0.1in}

\emph{The shift map}

Let $\Si$ be the set of the unilateral infinite sequences of zeros and ones. 
$$\Si=\{(s_i)_{i=1}^{\infty}; s_i\in\{0,1\}\}$$
Endowed with the metric $d(s,t)=\Si_{i=1}^{\infty}\frac{1}{2^{i}}|s_i-t_i|$, this is a compact metric space. 

		The map $\si:\Si\to\Si$ defined by $\si((s_i))=(s_{i+1})$ is called the shift map. It is well known that this map has topological entropy equals to $\log 2$.
\vspace {0.1in}

\emph{Expansiveness}

 We say that a map is positively expansive if there exists $e>0$ such that, if $d(f^i(x),f^i(y))<e$ for every $i\geq0$, then $x=y$. The number $e$ is called an expansiveness constant of $f$.

In \cite{RD} Reddy defined a pointwise version for expansiveness.

  \vspace{0.1in}
  \emph{We say that a point $x$ is a positively expansive point if there exists a constant $e>0$ such that  $$\{y\in X;\sup\limits_{i\geq0}d(f^i(x),f^i(y))<e\}=\{x\}.$$} 
  \vspace{0.1in}
However, as we shall see in section 5, even if all points are expansive points, we
can not be sure if the map f is expansive. .
 
 This motivate us to define what we called uniformly expansive point.  
  
  \vspace{0.1in}
  \emph{We say that a point $x$ is an uniformly positively expansive point if there exists a neighborhood $U$ of $x$ and a constant $e>0$ such that if $d(f^i(z),f^i(y))<e$ for every $j\geq 0$, whenever $y,z\in U$, then $y=z$.} 
  \vspace{0.1in}
  
Similarly, we define the following:
We say that a point is uniformly positive $n$-expansive, if  there exist a neighborhood $U$ of $x$ and $e>0$ such that if $y\in U$ then the cardinality of the set $$\{z\in U;\ \sup\limits_{i\geq 0}d(f^i(y),f^i(z))<e\}$$ is at most $n$.  

Finally, we say that a point $x$ is uniformly positive countable-expansive if  there exists a neighborhood $U$ of $x$ and $e>0$ such that for each  $y\in U$ 
the set $$\{z\in U;\  \sup\limits_{\i\geq 0}d(f^i(y),f^i(z))<e\}$$ is countable.

\begin{proposition} A map $f$ is positively expansive if, and only if , every point is positively uniformly expansive. 
\end{proposition}
\begin{proof}
	Obviously, we just need to prove the converse.
	If every point of $X$ is uniformly positively-expansive, then we can cover $X$ with a finite number of open sets $U_{x_1},...,U_{x_n}$ given by the expansiveness of the points $x_1,...,x_n$. Set $e=\min\{e_{x_1},...,e_{x_n},\eta\}$, where $\eta$ is the Lebesgue number of the covering. Clearly $e$ is a expansiveness constant for $f$.
\end{proof}

The previous proposition allows us to recover expansiveness by its uniform pointwise version. We remark that the previous proof applies to the n-expansive and countable expansive case.  

\vspace {0.1in}

\emph{Shadowing}

 A set $\{x_i\}_{i=0}^\infty$ is a $\delta$-pseudo-orbit for $f$, if $d(f(x_i),x_{i+1})<\delta$ for every $i$.
A $\de$-pseudo orbit $\{x_i\}_{i=0}^\infty$ is $\eps$-shadowed by $x$ if $d(f^i(x),x_i)<\eps$ for every $i$. We say that a map $f:X\to X$ has the shadowing property if for every $\eps>0$ there exists $\delta>0$ such that any $\delta$-pseudo-orbit is  $\eps$-shadowed by some point in $X$.

We say that a point $x_0$ is shadowable if for every $\eps>0$ there is a $\de>0$ such that every $\de$-pseudo-orbit $\{x_n\}_{n\geq0}$ satisfying $x=x_0$ is $\eps$-shadowed. 

\vspace{0.1in}

We say that a sequence of maps $f_n:X\to X$ has the uniform shadowing property if every $f_n$ has the shadowing property and if given $\eps>0$, all maps $f_n$  have the same shadowing constant $\delta$ related to $\eps$.

\section{Proof of Theorem \ref{exp}}

In this section we will give a proof for Theorem \ref{exp}. We start stating a result due to Kawaguchi, N. which exhibits an abstraction of the techniques developed in \cite{moop} by Moothathu  and Oprocha.

\begin{lemma}[Lemma 2.3 in \cite{kawaguchi}]\label{k1}
	Let $f : X \to X$ be a continuous map and let $x \in sh^+_b( f )$ with $b > 0$.
	Given $e,\de > 0$ , let $z$ be an accumulation point of the forward orbit of $x$. Suppose that the following conditions are satisfied.
	\begin{itemize}
		\item $e > 2b$.
	\item Every $\de$-pseudo orbit $\{x_i\}
	_{i=0}^\infty$ of $f$ with $x_0 = x$ is $b$-shadowed by some point of X.
	\item There is an $e$-separated pair $\{z_i^0\}_{i=0}^m$ and $\{z_i^1\}_{i=0}^m$
    of periodic $\de$-pseudo-orbits of $f$ through $z$ with period
	$m$.(i.e. $d(z_i^0,z_i^1)>e$ for $i=0,..,m$)
\end{itemize}
	Then there exists an entropy point $w$ of $f$ such that $d(x, w)\leq b$, furthermore $h( f )\geq \frac{\log 2}{m}$.
	
\end{lemma} 

In the previous lemma, by $sh_b^+(f)$ we mean the set of b-shadowable points (i.e. $x\in sh_b^+(f)$ if there exists $\de>0$ such that, every $\de$-pseudo-orbit through $x$ is $b$-shadowed by some point).

\begin{lemma}[Lemma 2.4 in \cite{kawaguchi}]\label{k2}
	
	Let $e \geq 2b > 0$. Let $\{z_i^0\}_{i=0}^m$ and $\{z_i^1\}_{i=0}^m$ form an $e$-separated pair
	of periodic $\de$-pseudo-orbits of $f$ through $z$ with period
	$m$. For each $s=(s_1,s_2,...)\in\Sigma$, define a $\de$-pseudo-orbit $\gamma(s) $
	as follows:
	$$\gamma (s) = \{z_0^{s_1},z_1^{s_1}...z_{m-1}^{s_1},z_0^{s_2},...,z_{m-1}^{s_2},z_0^{s_3},...\}$$
		If every $\gamma(s)$, $s\in \Sigma$, is $b$-shadowed by some point of $X$, then there exists a closed
	$f^m$-invariant subset $Y \subset$ $B_b(x)$ and a semi-conjugacy  map $\pi : (Y, f^m)\to (\Sigma,\sigma)$. 
	
\end{lemma}

In the previous lemma, $\sigma$ denotes the full shift map of two symbols. 

Now we point out the main idea on the proof of previous the lemmas. 
Suppose we have two periodic $\de$-pseudo-orbits $\{z_i^0\}_{i=0}^m$ and $\{z_i^1\}_{i=0}^m$ through $z$ such that $d(z_k^0,z_k^1)>e$. Suppose that any $\de$-pseudo-orbit through $z$ is $b$-shadowed. Let $b<\frac{e}{2}$. If $x_1$ and $x_2$ are real orbits which $b$-shadow $\{z_i^0\}_{i=0}^m$ and $\{z_i^1\}_{i=0}^m$ respectively, then it must be $x_1\neq x_2$. Thus for any two distinct sequences $s_1,s_2\in\Sigma$, points which  $b$-shadow $\gamma(s_1)$ and $\gamma(s_2)$ must be distinct. If we now define $Y\subset X $ as the set of all the $b$-shadows $x_{\gamma(s)}$ for the $\de$-orbits $\gamma(s)$, then we can define a surjection $\pi:Y\to \Sigma$ setting $\pi(x_{\gamma(s)})=s $. For more details about the proof of the previous lemmas, see \cite{kawaguchi}.

Let us be under the hypothesis of Theorem \ref{exp}. If we can exhibit two periodic $\de$-pseudo-orbits through $x$ , and $e,b>0$ satisfying the hypothesis of Lemmas \ref{k1} and \ref{k2}, we are almost done. If we reduce $B_e(x)$ to a neighborhood $U$ where the uniform expansiveness of $x$ is verified and set $e$ as the expansive constant of $x$, then expansiveness implies that any two points that $b$-shadow $\gamma(s)$ need to be equal,  i.e. the map $\pi$ of lemma \ref{k2} is a bijection. 

Now, let us proceed to construct suitable periodic pseudo-orbits for the lemmas.

Let $U$ be a neighborhood of $x$. Since $x$ is uniformly expansive, there are $e>0$ and an open neighborhood $V$ of $x$ such that for any pair of distinct points $y,z\in V$, there exists $l$ such that $d(f^l(x),f^l(y))>e$. 

Take $U'=U\cap V$ and let $0<b<\frac{e}{2}$ be such that $B_{b}(x)\subset U'$. Let $0<\de<\frac{b}{2}$ be given by the $\frac{b}{2}$-shadowing through $\{x\}$.

Since $X$ is compact, $f$ is uniformly continuous and therefore there is $0<\eta<\de$ such that $d(f(y),f(z))<\de$ when $d(y,z)<\eta$.
 
 Since $x$ is  non-periodic and nonwandering , we shall divide the proof in two cases.
 \vspace{0.1in}
 
\emph{Case 1: $x$ is a Recurrent Point:}

\vspace{0.1in}

In this case, there exists $m_1$ such that $f^{m_1}(x)\in B_{\eta}(x)$. Since $x$ is not isolated, the set $W=B_{\eta}(x)\setminus \{f(x),...,f^{m_1}(x)\}$ is an open neighborhood of $x$ and therefore there exists $m_2$ such that $f^{m_2}(x)\in W$. Let $x_p=f^{m_1}(x)$ and $x_q=f^{m_2}(x)$. Clearly $x_p\neq x_q$. Since $B_{\eta}(x)\subset V$, there exists $l$ such that $d(f^l(x_p),f^l(x_q))>e$.

Now, the set $W'=B_{\eta}(x)\setminus \{f(x),f^2(x),...,f^{m_2+l}(x)\}$ is an open neighborhood of $x$, then there exists $m_3>m_2+l$  such that $f^{m_3}(x)\in W'$. From a similar argument we can find $m_4>m_3$ such that $f^{m_4}(x)\neq f^{m_3}(x)$ and $f^{m_4}(x)\in W'$. Set $k_1=m_3-m_1$ and $k_2=m_4-m_2$.

 Consider the sets $$P=\{x,f(x_p),...,f^{k_1-1}(x_p),f^{k_1}(x_p)\} \text{ and } Q=\{x,f(x_q),...,f^{k_2-1}(x_q),f^{k_2}(x_q)\}$$
 
 We have $d(x,x_p)<\eta$ and $d(x,x_q)<\eta$. Therefore, $d(f(x),f(x_p))<\de$ and $d(f(x)f(x_q))<\de$. Thus the above sets are $\de$-pseudo-orbits through $\{x\}$. The shadowableness of $x$ implies that there exist $p$ and $q$ that $\frac{b}{2}$-shadow $P$ and $Q$ respectively. Thus $d(f^{nk_1}(p),x)<\frac{b}{2}$ and $d(f^{nk_2}(q),x)<\frac{b}{2}$ for every $n$. Then if $k=k_1k_2$ we have $f^{nk}(p),f^{nk}(q)\in U$ for every $n$. 
 
 To prove that $a\neq b$ we notice that $d(f^l(x_q),f^l(q))<b$, $d(f^l(x_p),f^l(p)),<b$ and $d(f^l(x_q),f^l(x_p))>e$. Thus $d(f^l(q),f^l(p))>\frac{e}{2}$ and $f^l(p)\neq f^l(q)$. Therefore $a\neq b$. Consider $N$  such that $m=Nk$ and $m>l$. 
 
 The sets $\{x,f(p),...,f^{m-1}(p),x\}$ and $\{x,b,f(q),...,f^{m-1}(q),x\}$ are periodic $\de$-pseudo-orbits through $x$. 
 
 Notice that we are under the hypothesis of Lemma \ref{k2} and we finish the proof in this case.

 \textit{Remark: If we are just interested in the positiveness of the entropy of $f$ we can apply directly lemma \ref{k1}.} 

\vspace{0.1in}

\emph{Case 2: x is not a recurrent point}
 
 \vspace{0.1in}
 
 If $x$ is not a recurrent point, there are $x_p\in B_{\eta}(x)$ and $k_1$ such that $x_p\neq x$ and $f^{k_1}(x_p)\in B_{\eta}(x)$. Since $x$ is not isolated, the set $W=B_{\eta}(x)\setminus \{x_p,f(x_p),...,f^{k_1}(x_p)\}$ is an open neighborhood of $x$. Thus there are $x_q$ and $k_2$ such that $f^{k_2}(x_q)\in W$.

Now, we proceed in the same way as in the previous case to obtain the points $p$ and $q$ satisfying suitable properties for Lemma \ref{k2} and again it gives us the desired conjugacy.

\section{Proof of Theorem \ref{nexp}}

Next we will prove theorem \ref{nexp} and corollary \ref{per}. We will adapt the previous proof to the current case. The main difference in this case is that one cannot guarantee that the map $\pi$ is a conjugacy. Indeed, $f$  can admit more than one $b$-shadow for the same pseudo-orbit $\gamma(s)$. Nevertheless, we are able to find suitable periodic pseudo-orbits for Lemmas $\ref{k1}$ and \ref{k2}.  
Suppose $x$ is not an isolated point. We can find any finite amount of points which behave in a way similar to the $p$ and $q$ in the previous proof. Now suppose $f$ is $n$-expansive. Then the previous remark gives us $p_1,...,p_{n+2}\in U $ and $k\in \N$ such that $f^{nk}(p_i)\in U$ for every $n$. Since $f$ is $n$-expansive, if we fix $p_i$ then there exists some $p_j$ such that $d(f^m(p_i),f^m(p_j))>\frac{e}{2}$ for some $m$. If we use $p_i$ and $p_j$  to construct the pseudo-orbits, the conclusion is just an application of Lemma \ref{k2}. 
\vspace{0.1in}


Finally, we give a proof for Corollary \ref{per}. 
Since $X$ is compact, the set of periodic points of $f$ must  accumulate at some point $x$. Since $f$ is an expansive map with the shadowing property, then $x$ is shadowable and expansive. Notice that $x$ is a nonwandering point, since it is accumulated by periodic points. In this case, $x$ could be a periodic point, but it does not matter. Indeed, by the proof of Theorem 7 we just need to find two distinct points with return to $U$, then we just take any two distinct periodic points in $U$. Hereafter, the proof is analogous to the the proof of Theorem \ref{exp}.

\section{Carvalho-Cordeiro's example}

In \cite{BW}, Carvalho and Cordeiro gave an example of a non-expansive but 2-expansive homeomorphism. Briefly, they start with a homeomorphism $f:(X,d)\to (X,d)$ which is expansive, has the shadowing property and dense periodic orbits $\{O_n\}_{n\geq 0}$. Then they add copies of these periodic orbits, obtaining a new compact metric space $(Y=X\cup \{O'_n\}, d')$.

Essentially, the new metric does $d'(O_n,O'_n)=1/n$. Also, the new map $g:Y\to Y$ is equal to $f$ on $X$ and it is periodic in each $O'_n$. This implies that for any $\eps$, if $1/n<\eps$ then the dynamical ball over $O_n$ has two orbits. Thus it is non expansive but is 2-expansive.

However, for any $n$, the orbits $O_n$ and $O'_n$ are $1/2n$-expansive points. The other points are clearly expansive, since $f$ is expansive. Thus, this example shows that even if any point is expansive the dynamics is not necessarily expansive.

Finally, Corollary \ref{per} implies that $f$ has positive entropy. Thus $g$ has positive entropy as well.

\section{Uniform Shadowing}

Now we are going to prove Theorem \ref{pentropy}. We begin with a lemma that slightly improves the result in \cite{fedeli} where it is assumed that $\delta_n(\eps)=\eps$ for every $n$.

\begin{lemma}
\label{t.unifshadow}\label{usha}
Let $X$ be a compact metric space and let $f_n:X\rightarrow X$ be a sequence of continuous functions converging uniformly to a map $f$. Suppose that $f_n$ has the uniform shadowing property. Then $f$ has the shadowing property.
\end{lemma}
\begin{proof}
	 Fix $\eps>0$. Since $f_n$ has the uniform shadowing property, chose $\de>0$ such that every $n$ and every $\de$-pseudo-orbit for $f_n$ is $\frac{\eps}{3}$-shadowed by some point $y_n$. We will show that every $\frac{\de}{2}$-pseudo-orbit for $f$ is $\eps$ shadowed.

	 Let $\{x_i\}$ be a $\frac{\de}{2}$-pseudo orbit of $f$. We claim that $\{x_i\}$
	 is a $\de$-pseudo orbit for $f_n$ if $n$ is sufficiently large. Indeed, since $f_n$ converges uniformly to $f$, there exists $N_0$ such that $d(f_n(x),f(x))<\frac{\de}{2}$ for every $x\in X$ and $n\ge N_0$. Then $d(f_n(x_i),x_{i+1})\le d(f_n(x_i),f(x_{i}))+d(f(x_i),x_{i+1})<\de$ if $n\ge N_0$.

	 Uniform shadowing implies that for every $n\ge N_0$ there exists a point $y_n$ such that $d(f_n^i(y_n),x_i)\le \frac{\eps}{3}$ for every $i$. Let $(y_n)_{n\ge N_0}$ be the sequence of such points. Since $X$ is compact we can assume that $y_n$ converges to some point $y\in X$. We claim that $y$ $\eps$-shadows $\{x_i\}$. Indeed, we have $$d(f^i(y),x_i)\le d(f^i(y),f^i(y_n))+d(f^i(y_n),f_n^i(y_n))+d(f_n^i(y_n),x_i)$$

	 Fix $i$. Since $f_n$ converges to $f$ then $f_n^i$ converges to $f^i$. On the other hand, $f^i$ is continuous. Then, we can make $d(f^i(y),f^i(y_n))\le \frac{\eps}{3}$ and $d(f^i(y_n),f_n^i(y_n))\le \frac{\eps}{3}$ if we chose $n$ sufficiently large. Therefore,  $d(f^i(y),x_i)\le\eps$ for every $i$. Thus $f$ has the shadowing property.

\end{proof}
We say that a map $f$ is nonwandering if every point in $X$ is a nonwandering point for $f$.

We say that a map is chain-recurrent if for any $\eps$ and any point $x$, there exists a finite $\eps$-pseudo-orbit $\{x_0,...,x_n\}$, such that $x_0=x_n=x$.
It is easy to see that any nonwandering map is chain-recurrent, and any chain-recurrent map with the shadowing property is nonwandering.

\begin{lemma}\label{nonw}
	Let $X$ be a compact metric space and $f_n:X\to X $ be a sequence of nonwandering maps converging uniformly to a map $f$. Suppose $f_n$ has the uniform shadowing property. Then $f$ is nonwandering.
\end{lemma} 

\begin{proof}
	Since $f_n$ is nonwandering, then it is chain-recurrent for every $n$. In the proof of previous lemma, we proved that any $\frac{\eps}{2}$-pseudo-orbit of $f_n$ is an $\eps$-pseudo of $f$, since $n$ is sufficiently large. Then $f$ is chain-recurrent and has the shadowing property. Therefore $f$ is nonwandering.
\end{proof}

In [5] the author proved that if $X$ is connected, $f$ is nonwandering and has the shadowing property then it is topologically mixing.

\vspace{0.1in}

\emph{Proof of Theorem \ref{pentropy}}
		
		By Lemma, \ref{nonw} $f$ is nonwandering. Since $X$ is connected, $f$ is topologically mixing. Therefore, $f$ is sensitive to initial conditions. Thus $f$ is a nonwandering map with the shadowing property and a sensitive point. Therefore $f$ has positive topological entropy (See \cite{moo}).

 \vspace{0.1in}

\emph{Remark}: We remark that we can prove in the same way as in lemma \ref{nonw}, that the uniform limit of a sequence of topologically transitive or topologically mixing maps with the uniform shadowing property, is topologically transitive or mixing respectively.

	\emph{Acknowledgment}: We would like to thank the referee for his valuable help in improving the presentation of this article.

\end{document}